\documentclass[12pt]{amsart}

\usepackage{amsmath}
\usepackage{amsfonts}
\usepackage{amssymb}
\usepackage{amsthm}
\usepackage{hyperref}
\usepackage{enumerate}
\usepackage{cleveref}
\usepackage{mathrsfs}
\usepackage[top=0.9in, bottom=1.15in, left=1.1in, right=1.1in]{geometry}
 \usepackage{hyperref}
\hypersetup{colorlinks=true,linkcolor=blue,citecolor=magenta}
\newtheorem{theorem}{Theorem}

\Crefname{conjecture}{Conjecture}{Conjectures}

\theoremstyle{plain}

\theoremstyle{plain}

\author{Robert Schneider}
\title{Infinite series for $\pi/3$ and other identities}
 
\begin{document}
 
\begin{abstract}
Using techniques from calculus, we combine classical identities for $\pi$, $\operatorname{ln}2$, and harmonic numbers, to arrive at a nice infinite series formula for $\pi/3$ that does not appear to be well known. 
In addition, we give twenty-seven related identities involving $\pi$ and other irrational numbers.
\end{abstract}

\maketitle

\centerline{\it In celebration of Pi Day 2022}
\

\section{Main identity and proof}
Recall the identity known in the literature as the Gregory--Leibniz formula for $\pi$ \cite{pi}:
\begin{equation}\label{Leibniz}
\frac{\pi}{4}\  =\  1-\frac{1}{3}+\frac{1}{5}-\frac{1}{7}+\frac{1}{9}-\dots .
\end{equation}
This identity is immediate from the Maclaurin series expansion of $\operatorname{arctan}x$ at $x=1$. Here, 
we prove another  infinite series formula that can be used to compute the value of $\pi$. 

\begin{theorem}\label{thm1}
We have the identity
\begin{equation*}\frac{\pi}{3}\  =\  \sum_{n=1}^{\infty}\frac{1}{n(2n-1)(4n-3)}. 
\end{equation*}
\end{theorem}

\begin{proof}
We begin with the Maclaurin series for the natural logarithm of 2: 
\begin{equation}\label{A}
\sum_{n=1}^{\infty}\frac{(-1)^{n+1}}{n}=\lim_{n\to\infty}\left( 1-\frac{1}{2}+\frac{1}{3}-\frac{1}{4}+-...+\frac{1}{2n-1}-\frac{1}{2n}  \right)=\operatorname{ln} 2.
\end{equation}
We may rewrite (\ref{A}) as 
$$\sum_{n=1}^{\infty}\left( \frac{1}{2n-1}-\frac{1}{2n} \right)= \sum_{n=1}^{\infty} \frac{1}{2n(2n-1)}=\operatorname{ln} 2,$$
thus
\begin{equation}\label{B}
\sum_{n=1}^{\infty} \frac{1}{4n(4n-2)}=\frac{1}{4}\sum_{n=1}^{\infty} \frac{1}{2n(2n-1)}=\frac{1}{4}\operatorname{ln} 2.
\end{equation}  
Euler \cite{Dunham} found that the difference between the $k$th harmonic number and $\operatorname{ln}k$ approaches a constant $\gamma = 0.5772...$ (the so-called Euler-Mascheroni constant), so we have
\begin{equation}\label{C}
\lim_{n \to \infty} \left( 1+\frac{1}{2}+\frac{1}{3}+\frac{1}{4}+...+\frac{1}{2n}-\operatorname{ln} (2n) \right)=\gamma.
\end{equation}
Adding limits from equations $(\ref{A})$ and $(\ref{C})$, then dividing by $2$ (and writing $\operatorname{ln}(2n)=\operatorname{ln}2+\operatorname{ln}n$), we find
$$\lim_{n\to\infty}\left( 1+\frac{1}{3}+\frac{1}{5}+...+\frac{1}{2n-1}-\frac{1}{2}(\operatorname{ln}2+\operatorname{ln} n) \right)=\frac{1}{2}(\operatorname{ln} 2 + \gamma),$$
thus
\begin{equation}\label{D}
\lim_{n \to \infty} \left( 1+\frac{1}{3}+\frac{1}{5}+...+\frac{1}{2n-1}-\frac{1}{2}\operatorname{ln} n \right)=\operatorname{ln} 2 + \frac{1}{2}\gamma.
\end{equation}
Splitting the sum on the left in half and reorganizing, equation (\ref{D}) may be rewritten 
\begin{equation}\label{E}
\lim_{n \to \infty} \left( \frac{1}{n+1}+\frac{1}{n+3}+...+\frac{1}{2n-1}\right)
\end{equation}
\begin{equation*}
=\operatorname{ln} 2 + \frac{1}{2}\gamma-\lim_{n \to \infty} \left( 1+\frac{1}{3}+\frac{1}{5}+...+\frac{1}{n-1}-\frac{1}{2}\operatorname{ln} n \right).
\end{equation*}
Now, substitute $n/2$ for $n$ in (\ref{D}), then 
subtract $1/2\cdot \operatorname{ln} 2$ from  both sides, to arrive at 
\begin{equation}\label{F}
\frac{1}{2}\operatorname{ln} 2 + \frac{1}{2}\gamma=\lim_{n \to \infty} \left( 1+\frac{1}{3}+\frac{1}{5}+...+\frac{1}{n-1}-\frac{1}{2}\operatorname{ln} n \right).
\end{equation}
Adding the corresponding sides of $(\ref{E})$ and $(\ref{F})$, 
then subtracting $1/2\cdot \operatorname{ln} 2+\gamma/2$ from the resulting equation, gives
\begin{equation}\label{G}
\lim_{n \to \infty} \left( \frac{1}{n+1}+\frac{1}{n+3}+...+\frac{1}{2n-1}\right)=\frac{1}{2}\operatorname{ln} 2.
\end{equation}
Next we will use Leibniz's formula (\ref{Leibniz}), which we can write in the form 
\begin{equation}\label{H}
\lim_{n\to\infty}\left( 1-\frac{1}{3}+\frac{1}{5}-\frac{1}{7}+-...+\frac{1}{2n-3}-\frac{1}{2n-1}  \right)=\frac{\pi}{4}.
\end{equation}
If we add the limits from $(\ref{D})$ and $(\ref{H})$, then  
divide both sides by $2$, we find 
$$\lim_{n\to\infty}\left( 1+\frac{1}{5}+\frac{1}{9}+...+\frac{1}{2n-3}-\frac{1}{4}\operatorname{ln} n \right)=\frac{\pi}{8}+\frac{1}{2}\left(\operatorname{ln} 2 + \frac{1}{2}\gamma\right).$$
Subtracting $1/2$ times the limit in (\ref{D}) from this equation gives
\begin{flalign*}
\lim_{n\to\infty}\left(1-\frac{1}{2}+\frac{1}{5}-\frac{1}{6}+-...+\frac{1}{2n-3}-\frac{1}{2n-2}\right)
-&\frac{1}{2} \left( \frac{1}{n+1}+\frac{1}{n+3}+...+\frac{1}{2n-1}\right)\\&=\frac{\pi}{8}.
\end{flalign*}
Adding $1/2$ times equation (\ref{G}) to both sides of this expression yields
\begin{equation*}
\lim_{n\to\infty}\left(1-\frac{1}{2}+\frac{1}{5}-\frac{1}{6}+-...+\frac{1}{2n-3}-\frac{1}{2n-2}\right)
=\frac{\pi}{8}+\frac{1}{4}\operatorname{ln} 2,
\end{equation*}
which can be rewritten 
\begin{equation}\label{J}
\sum_{n=1}^{\infty}\left(\frac{1}{4n-3}-\frac{1}{4n-2}\right)=\sum_{n=1}^{\infty}\frac{1}{(4n-2)(4n-3)}=\frac{\pi}{8}+\frac{1}{4}\operatorname{ln} 2.
\end{equation}
Subtracting equation (\ref{B}) from (\ref{J}) yields  
\begin{equation*}
\sum_{n=1}^{\infty}\left(\frac{1}{(4n-2)(4n-3)}-\frac{1}{4n(4n-2)}\right)=\sum_{n=1}^{\infty}\frac{3}{4n(4n-2)(4n-3)}=\frac{\pi}{8}.
\end{equation*}
Finally, multiplication by $8/3$ gives the theorem.
\end{proof}
%
%
%
%

\

\section{Further identities}
During the writing of this paper (2006), the author also found a large number of related summation identities by combining Theorem \ref{thm1} with other equations in the preceding proof, 
along with well-known zeta values such as Euler's beautiful formula (see  \cite{Dunham}),
$$\sum_{n=1}^{\infty}\frac{1}{n^2}\  =\  \frac{\pi^2}{6},$$
and other classical summation identities from \cite{Jolley}. Below, there is a selection of these identities given without proof, loosely organized by number of factors in the denominators of the summands. The first identity follows from the telescoping series $(1-1/2)+(1/2-1/3)+(1/3-1/4)+...$, the second and third identities follow by rewriting  (\ref{A}) and (\ref{H}), respectively, and the rest arise from liberal use of partial fractions and recursion relations. 

The interested reader might like to prove the following results; certainly further identities like these can be discovered: 

%
%

%

%

\begin{equation}
\sum_{n=1}^{\infty}\frac{1}{(2n-1)(2n+1)}=\frac{1}{2}
\end{equation}

\begin{equation}
\sum_{n=1}^{\infty}\frac{1}{(n+1)(2n+1)}=2\operatorname{ln} 2
\end{equation}

\begin{equation}
\sum_{n=1}^{\infty}\frac{1}{(4n-1)(4n-3)}=\frac{\pi}{8}
\end{equation}

\begin{equation}
\sum_{n=1}^{\infty}\frac{1}{(2n-1)(4n-3)}=\frac{\pi+2\operatorname{ln}2}{4}
\end{equation}
\begin{equation}
\sum_{n=1}^{\infty}\frac{1}{n(4n-3)}=\frac{\pi+6\operatorname{ln}2}{6}
\end{equation}

\begin{equation}
\sum_{n=1}^{\infty}\frac{1}{n(4n-1)}=\frac{6\operatorname{ln}2-\pi}{2}
\end{equation}

\begin{equation}
\sum_{n=1}^{\infty}\frac{1}{(2n-1)(4n-1)(4n-3)}=\frac{\operatorname{ln}2}{2}
\end{equation}

\begin{equation}
\sum_{n=1}^{\infty}\frac{1}{n(n+1)(2n+1)}=\frac{2\operatorname{ln} 2-1}{2}
\end{equation}

\begin{equation}
\sum_{n=1}^{\infty}\frac{1}{(4n+1)(4n-1)(4n-3)}=\frac{\pi-2}{16}
\end{equation}

\begin{equation}
\sum_{n=1}^{\infty}\frac{1}{n^2(2n-1)}=\frac{24\operatorname{ln}2-\pi^2}{6}
\end{equation}

\begin{equation}
\sum_{n=1}^{\infty}\frac{1}{n^2(4n-1)}=\frac{72\operatorname{ln}2-12\pi-\pi^2}{6}
\end{equation}

\begin{equation}
\sum_{n=1}^{\infty}\frac{1}{n(n+1)^2}=\frac{12-\pi^2}{6}
\end{equation}

\begin{equation}
\sum_{n=1}^{\infty}\frac{1}{n^2(2n+1)}=\frac{\pi^2+24\operatorname{ln} 2 - 24}{6}
\end{equation}

\begin{equation}
\sum_{n=1}^{\infty}\frac{1}{n^2(n+1)}=\frac{\pi^2-6}{6}
\end{equation}

\begin{equation}
\sum_{n=1}^{\infty}\frac{1}{n^2(n+1)^2}=\frac{\pi^2-9}{3} 
\end{equation}

\begin{equation}
\sum_{n=1}^{\infty}\frac{1}{(2n+1)(2n-1)(4n+1)(4n-1)}=\frac{\pi-3}{6}
\end{equation}

\begin{equation}
\sum_{n=1}^{\infty}\frac{1}{(4n+1)(4n-1)(8n+1)(8n-1)}=\frac{\pi (1+2\sqrt{2})-12}{24}
\end{equation}

\begin{equation}
\sum_{n=1}^{\infty}\frac{1}{n(2n-1)(4n-1)(4n-3)}=\frac{6\operatorname{ln}2-\pi}{3}
\end{equation}

\begin{equation}
\sum_{n=1}^{\infty}\frac{1}{n^2(n+1)(2n+1)}=\frac{\pi^2-12\operatorname{ln} 2}{6}
\end{equation}

\begin{equation}
\sum_{n=1}^{\infty}\frac{1}{n^2(2n-1)(4n-3)}=\frac{\pi^2+8\pi-24\operatorname{ln}2}{18}
\end{equation}

\begin{equation}
\sum_{n=1}^{\infty}\frac{1}{n^2(2n-1)(4n-1)}=\frac{\pi^2+24\pi-120\operatorname{ln}2}{6}
\end{equation}

\begin{equation}
\sum_{n=1}^{\infty}\frac{1}{n^2(2n-1)(4n-1)(4n-3)}=\frac{168\operatorname{ln}2-32\pi-\pi^2}{18}
\end{equation}

\begin{equation}
\sum_{n=1}^{\infty}\frac{1}{n(2n+1)(2n-1)(6n+1)(6n-1)}=\frac{52-32\operatorname{ln}2-27\operatorname{ln}3}{16}
\end{equation}

\begin{equation}
\sum_{n=1}^{\infty}\frac{1}{n(2n+1)(2n-1)(3n+1)(3n-1)}=\frac{19+16\operatorname{ln}2-27\operatorname{ln}3}{10}
\end{equation}

\begin{equation}
\sum_{n=1}^{\infty}\frac{1}{n^3(n+1)^3}=10-\pi^2
\end{equation}

\begin{equation}
\sum_{n=1}^{\infty}\frac{1}{(2n+1)(2n-1)(4n+1)(4n-1)(8n+1)(8n-1)}=\frac{45-\pi (3+8\sqrt{2})}{6}
\end{equation}

\begin{equation}
\sum_{n=1}^{\infty}\frac{1}{n(2n+1)(2n-1)(3n+1)(3n-1)(6n+1)(6n-1)}=\frac{64\operatorname{ln}2+27\operatorname{ln}3-74}{20}
\end{equation}


\  

\section*{Acknowledgments}
Dedicated to my son, mathematician, programmer and musician, Maxwell Schneider. 
This article is a LaTeX transcription of the author's first mathematics paper, written in 2006 at a time when I was almost entirely self taught in mathematics. I am deeply thankful to Paul Blankenship of Bluegrass Community and Technical College for teaching me calculus, encouraging my efforts in analysis and checking my proof of Theorem \ref{thm1}; and to David Leep at the University of Kentucky for mentoring me when I had little formal training in mathematics, and for urging me to write up these results. 

\end{document}